\documentclass{patmorin}
\usepackage[utf8]{inputenc}
\usepackage{amsthm,amsmath,graphicx}
\usepackage{array}
\usepackage{pat}
\usepackage{hyperref}
\usepackage[dvipsnames]{xcolor}
\definecolor{linkblue}{named}{Blue}
\hypersetup{colorlinks=true, linkcolor=linkblue,  anchorcolor=linkblue,
citecolor=linkblue, filecolor=linkblue, menucolor=linkblue,
urlcolor=linkblue, pdfcreator=Me, pdfproducer=Me} 
\usepackage{tikz}

\DeclareMathOperator{\block}{block}

\title{\MakeUppercase{New Bounds for Facial Nonrepetitive Colouring}\thanks{This research is partially funded by NSERC and the Ontario Ministry of Research and Innovation.}}

\author{Prosenjit Bose,\thanks{School of Computer Science, Carleton University}\,\, Vida Dujmovi\'c,\thanks{Department of Computer Science and Electrical Engineering, University of Ottawa}\,\, Pat Morin,\footnotemark[2]\,\, Lucas Rioux-Maldague\thanks{Google}}

\begin{document}
\maketitle

\begin{abstract}
  We prove that the facial nonrepetitive chromatic number of any outerplanar graph is at most 11 and of any planar graph is at most 22.
\end{abstract}

\section{Introduction}

A sequence $S=s_1,s_2,\cdots,s_{2r}$, $r\ge 1$, is a \emph{repetition}
if $s_i=s_{r+i}$ for each $i\in\{1,\ldots,r\}$. For example, 1212
is a repetition, while 1213 is not. A \emph{block} of a sequence
$S$ is any subsequence of consecutive terms in $S$. A sequence
is \emph{nonrepetitive} if for every non-empty block $B$ of $S$,
$B$ is not a repetition. Otherwise $S$ is \emph{repetitive}. For
example, 1312124 is repetitive as it contains the block 1212 which is a
repetition, while 123213 is nonrepetitive as it contains no such block.
No sequence of length greater than three using only two symbols can be
nonrepetitive. A result by Axel Thue in 1906 states that nonrepetitive
sequences of infinite length can be created using three symbols
\cite{thue1906uber}. Thue's work is considered to have initiated the
study of the combinatorics of words \cite{allouche1999ubiquitous}.

A graph colouring variation on this theme was proposed by Alon
\etal\ \cite{alon2002nonrepetitive}.  A nonrepetitive (vertex) colouring
of a graph $G$ is an assignment of colours to the vertices of $G$ such
that, for every path $P$ in $G$, the sequence of colours of vertices
in $P$ is not a repetition.  The \emph{nonrepetitive chromatic number}
of $G$, denoted $\pi(G)$, is the minimum number of colours required to
nonrepetitively colour $G$.  In this setting, Thue's result states that,
for all $n\ge 1$, the path $P_n$ on vertices has $\pi(P_n)\le 3$.
Since its introduction, nonrepetitive
graph colouring has received much attention \cite{barat2013facial,
barat2007square, barat2008note, brevsar2007nonrepetitive,
currie2002cycle18, dujmovic2012planarlogn, dujmovic2011nonrepetitive,
fiorenzi2011thue, gagol2016pathwidth, gonccalves2014entropy,
grytczuk2007nonrepetitivesurvey, grytczuk2007nonrepetitive,
grytczuk2013new, harant2012nonrepetitive, kozik2013nonrepetitive,
kundgen2008nonrepetitive, pezarski2009non, przybylo2013facial,
schreyer2012facial, schreyer2013total}.

A well-known conjecture, due to Alon \etal\ \cite{alon2002nonrepetitive},
is that there exists a constant $K$ such that, for every planar graph
$G$, $\pi(G) \leq K$.  The current best upper bound for $n$-vertex planar
graphs is $O(\log n)$ \cite{dujmovic2012planarlogn}. No planar graph with
nonrepetitive chromatic number greater than 11 is known (see Appendix~A
in \cite{dujmovic2012planarlogn}).

More is known about the \emph{facial} version of the problem for embedded
planar graphs.  Harant and Jendrol \cite{harant2012nonrepetitive} asked
if every plane graph can be coloured with a constant number of colours
such that every \emph{facial} path\footnote{A facial path is a path that
is a contiguous subsequence of a facial walk; see \secref{preliminaries}
for a more rigorous definition.} is nonrepetitively coloured.  Barát and
Czap \cite{barat2013facial} answered this question in the affirmative
by showing that $24$ colours are sufficient.  We reduce this bound
to 22 by proving a bound of 11 for facial nonrepetitive colouring of
outerplane graphs.

\subsection{Related Work}

\paragraph{Nonrepetitive Colouring.}

It is known that some families of graphs have bounded
nonrepetitive chromatic number.  In their original work,
Alon \etal\ \cite{alon2002nonrepetitive} showed that $\pi(G) =
O(\Delta^2)$ if $G$ has maximum degree $\Delta$ and that there
are are graphs of maximum degree $\Delta$ with nonrepetitive
chromatic number $\Omega(\Delta^2/\log \Delta)$.  The constants
in the $O(\Delta^2)$ upper bound have been steadily improved
\cite{dujmovic2011nonrepetitive,grytczuk2007nonrepetitivesurvey,grytczuk2007nonrepetitive,harant2012nonrepetitive}.

Barát and Varjú \cite{barat2007square} and K{\"u}ndgen and Pelsmajer
\cite{kundgen2008nonrepetitive} independently showed that $\pi(G)\le
12$ if $G$ is outerplanar and, more generally, $\pi(G)\le c^t$ if $G$
has treewidth at most $t$.  (Barát and Varjú proved the latter bound
with $c=6$ while K{\"u}ndgen and Pelsmajer proved it with $c=4$.)
The bound of $4^t$ for $t$-trees is tight if $t=1$ (trees), but it
is not known if it is tight for other values of $t$. Even the upper
bound of 12 for outerplanar graphs may not be tight, as no outerplanar
graph with nonrepetititive chromatic number greater than 7 is known
\cite{barat2007square}.

\paragraph{Facial Nonrepetitive Colouring.}

Facial nonrepetitive colouring was first considered by Havet et
al \cite{havet2011facial}, who studied the edge-colouring variant
of the problem.  In this setting, they were able to show that
the edges of any plane graph can be 8-coloured so that every facial
trail\footnote{A \emph{facial trail} is a contiguous subsequence of the
edges traversed during the boundary walk of a face.} is coloured
nonrepetitively.  For the list-colouring version of this problem,
Przyby{\l}o \cite{przybylo2013facial} showed that lists with at least
12 colours are sufficient to colour the edges of any plane graph so that
every facial trail is coloured nonrepetitively.

For the vertex-colouring version we study, Harant and Jendrol
\cite{harant2012nonrepetitive} proved that $\pi_f(G)=O(\log\Delta)$ if
$G$ is a plane graph of maximum degree $\Delta$ and that $\pi_f(G)\le
16$ if $G$ is a Hamiltonian plane graph.  They also conjectured that
$\pi_f(G)=O(1)$ when $G$ is any plane graph.  As mentioned above, this
latter conjecture was confirmed by Barát and Czap \cite{barat2013facial},
who showed that, for any plane graph $G$, $\pi_f(G)\le 24$.  The results
of Barát and Czap \cite{barat2013facial} also extend to graphs embedded
in surfaces.  They show that a graph embedded on a surface of genus $g$
can be facially nonrepetitively $24^{2g}$-coloured.  The best lower
bounds for facial nonrepetitive chromatic numbers are 5 for plane graphs
and 4 for outerplane graphs \cite{barat2013facial}.

\section{Preliminary Results and Definitions}
\seclabel{preliminaries}

We assume the reader is familiar with standard graph theory terminology
as used by, e.g., Bondy and Murty \cite{bondy.murty:graph}.  All graphs
we consider are undirected, but not necessarily simple; they may contain
loops and parallel edges.  For a graph, $G$, we use the notations $V(G)$
and $E(G)$ to denote $G$'s vertex and edge sets, respectively. For
$S\subset V(G)$, $G[S]$ denotes the subgraph of $G$ induced by the
vertices in $S$ and $G-S=G[V(G)\setminus S]$.

A graph is \emph{$k$-connected} if it contains more than $k$ vertices
and has no vertex cut of size less than $k$.  A \emph{$k$-connected
component} of a graph $G$ is a maximal subset of vertices of $G$ that
induces a $k$-connected subgraph.  A \emph{bridge} in a graph $G$ is an
edge whose removal increases the number of connected components.
A graph is \emph{bridgeless} if it has no bridges.

A \emph{plane graph} $G$ is a fixed embedding of a graph in the plane
such that its edges intersect only at their common endpoints. An
\emph{outerplane} graph $G$ is a plane graph such that all the
vertices of $G$ are incident on the outer face of $G$. A \emph{chord}
in an outerplane graph is an edge that is not incident to the outer
face. A \emph{cactus graph} is an outerplane graph with no chords.
An \emph{ear} in an simple outerplane graph is an inner face that is
incident to exactly one chord. An ear is \emph{triangular} if it has
exactly three vertices.  The \emph{weak dual} of a outerplane graph $G$
is a forest whose vertices are the inner faces of $G$ and that contains
the edge $fg$ if the face $f$ and the face $g$ have a chord in common.
Note that the ears of $G$ are leaves in the weak dual and that, if $G$
is biconnected, then its weak dual is a tree.

A \emph{walk} in a graph $G$ is a sequence of vertices
$v_0,\ldots,v_{\ell-1}$ such that, for every $i\in\{0,\ldots,\ell-2\}$
the edge $v_iv_{i+1}$ is in $E(G)$.  The walk is \emph{closed} if
$v_{\ell-1}v_0$ is also in $E(G)$.  A walk is a \emph{path} if all
its vertices are distinct.  A \emph{facial walk} in a plane graph
$G$ is a closed walk $v_0,\ldots,v_{\ell-1}$ such that, for every
$i\in\{0,\ldots,\ell-1\}$, the edges $v_{(i-1)\bmod \ell} v_i$ and
$v_iv_{(i+1)\bmod\ell}$ occur consecutively in the counterclockwise
cyclic ordering of the edges incident to $v_i$ in the embedding of $G$.
A \emph{facial path} is a contiguous subsequence of a facial walk that
is also a path in $G$.  A facial path is an \emph{outer-facial path}
if it appears in a facial walk of the outer face of $G$ and it is an
\emph{inner-facial} path if it appears in a facial walk of some inner
face of $G$.

Before proceeding with our results, we introduce a helper lemma
due to Havet \etal~\cite{havet2011facial} and two theorems that
will be used throughout the paper. The helper lemma provides a way to
interlace nonrepetitive sequences.

\begin{lem}[Havet \etal~\cite{havet2011facial}]\lemlabel{interleave}
  Let $B=B_1,B_2,\ldots,B_k$ be a nonrepetitive sequence over an alphabet
  $\mathcal{B}$ in which each $B_i$ has size at least 1. For each $i
  \in \{0,\ldots,k\}$, let $A_i$ be a (possibly empty) nonrepetitive
  sequence over an alphabet $\mathcal{A}$ with $\mathcal{B} \cap
  \mathcal{A} = \emptyset$. Then $S = A_0, B_1, A_1, \ldots, B_k, A_k$
  is a nonrepetitive sequence.
\end{lem}

We will require two results about the nonrepetitive chromatic number of
trees and cycles:

\begin{thm}[Alon \etal~\cite{alon2002nonrepetitive}]\thmlabel{tree}
  For every tree, $T$, $\pi(T) \leq 4$.
\end{thm}

\begin{thm}[Currie \cite{currie2002cycle18}]\thmlabel{cycle}
  For every $n>2$, the cycle $C_n$ on $n$ vertices has
  \[
  \pi(C_n) = \begin{cases}
              4 & \text{ if } n \in \{5,7,9,10,14,17\} \\
              3 & \text{ otherwise. }
             \end{cases}
  \]
\end{thm}

\section{Outerplane Graphs}

We begin with a simple lemma that allows us to focus, when convenient,
on simple outerplane graphs.

\begin{lem}\lemlabel{non-simple}
  Let $G$ be a simple outerplane graph and let $G'$ be an outerplane
  graph obtained by adding parallel edge and/or loops to $G$.  Then, any
  facially nonrepetitive colouring of $G$ is also a facially nonrepetitive
  colouring of $G'$, so $\pi_f(G')\le \pi_f(G)$.
\end{lem}

\begin{proof}
   We argue that any facial path (described as a sequence of vertices)
   in $G'$ is also a facial path in $G$.  Therefore, by facially
   nonrepetitively colouring $G$, we obtain a facial nonrepetitive
   colouring of $G'$.

   First, note that no facial path uses a loop, so the addition of loops
   does not introduce new facial paths in $G'$.  When a (non-loop) edge
   $e'$ is added parallel to an existing edge $e$ of $G$, the union
   of the embeddings of $e$ and $e'$ form a Jordan curve that does not
   contain any vertices of $G$ (since we require $G'$ to be outerplane).
   This implies that any facial path in $G'$ that uses the new edge $e'$
   exists in $G$ as a facial path that uses the edge $e$.
\end{proof}

Next, we introduce a definition that is crucial to the rest of the paper.
Let $G$ be an outerplane graph. A \emph{blocking set} of $G$ is a set
of vertices $B \subseteq V(G)$ such that for each 2-connected component
$H$ of $G$, $H-B$ is a tree and for each inner face $F$,
$V(F) \setminus B \ne \emptyset$.  See Figure \ref{fig:blocking-set}
for an example of a blocking set.

\begin{figure}
  \begin{center}
     \includegraphics{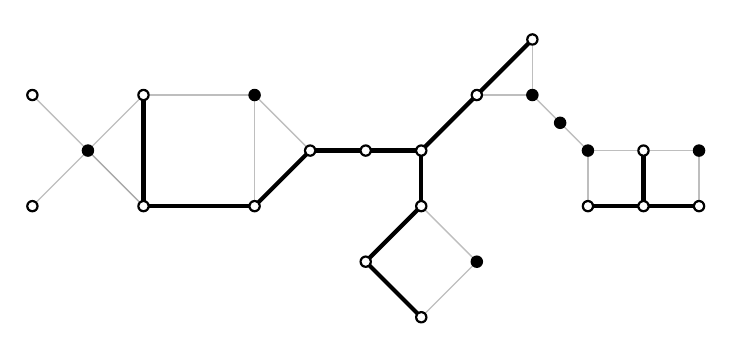}
  \end{center}
  \caption{The blocking set, $B$ (black vertices), of an outerplane graph,
     $G$,  (with the edges of $G-B$ shown in bold).}
  \figlabel{blocking-set}
\end{figure}

The definition of a blocking set is subtle and implies some properties
that we will use throughout.   

\begin{obs}\obslabel{no-chords}
   For any blocking set $B$ of $G$, $B$ does not include both endpoints
   of any chord $c$ of $G$. 
\end{obs}

\begin{obs}\obslabel{consecutive}
   For any blocking set $B$ of $G$ and any inner face $F$ of $G$, the
   vertices of $V(F)\cap B$ occur consecutively on the boundary of $F$. 
   In other words, $F-B$ is a non-empty path.  
\end{obs}

Observations~\ref{obs:no-chords} and \ref{obs:consecutive} are true
because, otherwise, $H-B$ would be disconnected for the 2-connected
component containing $c$ or $F$, respectively.


\begin{lem}\lemlabel{biconnected}
  For every biconnected outerplane graph, $G$, and any vertex $v\in
  V(G)$, there exists a blocking set $B$ of $G$ such that $v\in B$ and,
  for each inner face $F$ of $G$, $|B\cap V(F)|=1$.
\end{lem}

\begin{proof}
  The proof is by induction on the number of inner faces.  If $G$ has
  only one inner face, we take $B=\{v\}$ and we are done.  Otherwise,
  select some ear, $F$ of $G$ whose chord is $uw$ and such that
  $v\not\in V(F)\setminus\{u,w\}$. Such an ear $F$ always exists because
  $G$ has at least two ears.  Let $G'=G-(V(F)\setminus\{u,w\})$.
  The graph $G'$ has one less inner face than $G$ so, by induction, it
  has a blocking set $B'$ that satisfies the conditions of the lemma.
  There are two cases to consider:
  \begin{enumerate}
    \item If one of $u$ or $w$ is in $B'$ then we take $B=B'$ to obtain
      a blocking set that satisifes the conditions of the lemma.

    \item Otherwise, let $x$ be any vertex in $V(F)\setminus\{u,w\}$ and
      take $B=B'\cup\{x\}$ to obtain a blocking set that satisifes the
      conditions of the lemma. \qedhere
  \end{enumerate}
\end{proof}

\Lemref{biconnected} allows us to prescribe that a particular vertex
$v$ be included in the blocking set, but it will also be convenient to
exclude a particular vertex $v$ by using \lemref{biconnected} to force
the inclusion of $v$'s neighbour on the outer face (which is also on
some inner face with $v$).

\begin{cor}\corlabel{biconnected-out}
  For every biconnected outerplane graph, $G$, and any vertex $v\in
  V(G)$, there exists a blocking set $B$ of $G$ such that $v\not\in B$ and,
  for each inner face $F$ of $G$, $|B\cap V(F)|=1$.
\end{cor}

At this point we pause to sketch how \lemref{biconnected}
can already be used to give an upper-bound of 8 on the facial
nonrepetitive chromatic number of biconnected outerplane graphs.  For a
biconnected outerplane graph, $G$, we take a blocking set $B$ of $G$
using \lemref{biconnected}.  By \thmref{tree}, we can nonrepetitively
4-colour the tree $T=G-B$ using the colours $\{1,2,3,4\}$, so
what remains is to assign colours to the vertices in $B$.  To do this,
we use \thmref{cycle} to nonrepetitively 4-colour the cycle, $C$, that
contains the vertices of $B$ in the order they appear on the outer face
of $G$ using the colours $\{5,6,7,8\}$.  We claim that the resulting
8-colouring of $G$ is facially nonrepetitive.  No facial path on an
inner face is coloured repetitively since each such facial path is also
either present in the tree $T$ or it contains exactly one vertex of $B$.
No facial path on the outer face is coloured repetitively since it is
obtained by interleaving a nonrepetitive sequence of colours in $C$
with nonrepetitive sequences taken from $T$; by \lemref{interleave},
a sequence obtained in this way is nonrepetitive.

In \appref{biconnected}, we show that the preceding argument can be
improved to give a bound of 7 on the facial nonrepetitive chromatic
number of biconnected outerplane graphs. This is just a matter of adding
vertices to the blocking set so that the cycle $C$ does not have length
in $\{5,7,9,10,14,17\}$, so that it can be nonrepetitively 3-coloured.

Finally, we remind the reader that, although \lemref{biconnected} and
\corref{biconnected-out} provide blocking sets that include only one
vertex on each inner face, not all blocking sets have this property.
It is helpful to keep this in mind in the next section.

\subsection{The Blocking Graph}

The \emph{blocking graph} of $G$ for a blocking set $B$ is the graph,
denoted by $\block_B(G)$, whose vertex set is $B$ and whose edges are
defined as follows: Begin with the (closed) facial walk $W$ on the outer
face of $G$. Remove every vertex not in $B$ from $W$ to obtain a cyclic
sequence $W'$ of vertices in $B$. For each consecutive pair of vertices
$uw$ in $W'$ we add an edge $uw$ to $\block_B(G)$.  This naturally
defines the embedding of $\block_B(G)$. See \figref{blocking-graph}
for an example.   Note that $\block_B(G)$ is a plane graph that is not
necessarily simple; it may contain parallel edges (cycles of length
two) and self-loops (cycles of length one).

\begin{figure}
  \begin{center}
     \includegraphics{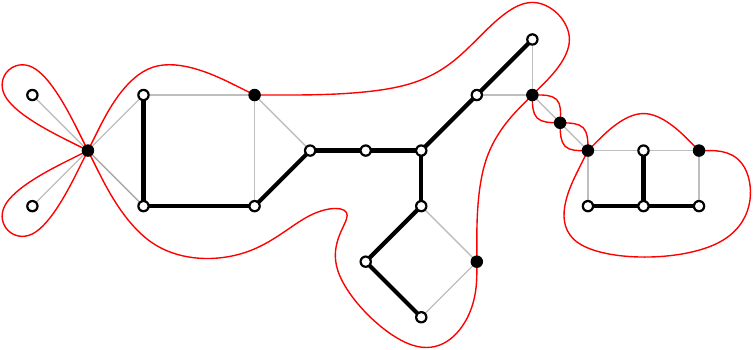}
  \end{center}
  \caption{The blocking graph (curved edges in red) associated with a
           blocking set.}
  \figlabel{blocking-graph}
\end{figure}

The fact that a blocking set does not contain both endpoints of any
chord of $G$ (\obsref{no-chords}) implies the following observation:

\begin{obs}\obslabel{cactus}
   For every outerplane graph $G$ and any blocking set $B$ of $G$,
   the blocking graph $\block_B(G)$ is a bridgeless cactus graph.
\end{obs}

\begin{obs}\obslabel{shit0}
  For every outerplane graph $G$, any blocking set $B$ of $G$, and any
  facial path $P$ on the outer face of  $G$, the subsequence of $P$
  containing only the vertices of $B$ is a (outer) facial path in
  $\block_B(G)$.
\end{obs}

\begin{obs}\obslabel{shit}
  For every outerplane graph $G$, any blocking set $B$ of $G$ and any
  inner face $F$ of $G$, $G[V(F)\cap B]$ is a non-empty path that is a
  facial path (on some inner face)  in $\block_B(G)$.
\end{obs}

In the previous section, we sketched a proof of an upper bound of 8 on
the facial chromatic number of biconnected outerplane graphs.  This proof
works by nonrepetitively 4-colouring the tree, $G-B$, obtained after
removing the blocking set and then nonrepetitively 4-colouring a cycle,
$C$, of vertices in the blocking set.  This cycle, $C$, is actually
the blocking graph, $\block_B(G)$.  The following lemma shows that
this strategy generalizes to the situation where we can find a facial
nonrepetitive colouring of $\block_B(G)$ with few colours.

\begin{lem}\lemlabel{hitting_plus_four}\lemlabel{k-plus-four}
  Let $G$ be an outerplane graph and $B$ be a blocking set of $G$. If
  there exists a facial nonrepetitive $k$-colouring of (the outer 
  face of) $\block_{B}(G)$, then there exists a facial nonrepetitive
  $(4+k)$-colouring of $G$.
\end{lem}

\begin{proof}
By \thmref{tree} we can colour $G-B$ nonrepetitively using colours
$\{1,2,3,4\}$ and, by assumption, we can facially nonrepetitively colour
$\block_{B}(G)$ with colours $\{5,\dots, k+4\}$. These two colourings
define a colouring of $G$ that we now show is facially nonrepetitive.

Let $P$ be a facial path in $G$. If $P$ is a facial path of
$\block_{B}(G)$ or a path in $G-B$ then there is nothing to prove.

Otherwise consider first the case that $P$ is a path on an inner
face $F$ of $G$.   There are two cases to consider:
\begin{enumerate}
\item The colouring of $P$ is of the form $A_0,B_1,A_1$, where
  $A_0$ and $A_1$ are obtained from (possibly-empty) paths in $G-B$
  and $B_1$ is obtained from a non-empty outer-facial path in $\block_B(G)$
  (by \obsref{shit}).  \lemref{interleave} therefore implies that the
  colour sequence $A_0,B_1,A_1$ is nonrepetitive.

\item The colouring of $P$ is of the form $B_0,A_1,B_1$, where $A_1$
  is obtained from a non-empty path in $G-B$ and $B_0$ and $B_1$ are
  (possibly empty) outer-facial paths in $\block_B(G)$ (by \obsref{shit}).
  Again, \lemref{interleave} implies that the resulting colour sequence
  is nonrepetitive.
\end{enumerate}

Finally, consider the case where $P$ is a facial path on the outer face
of $G$.  In this case, the colour sequence obtained from $P$ is of the
form $A_0,B_1,A_1,\ldots,B_k,A_k$ where each $A_i$ is obtained from a
(possibly empty) path in $G-B$ and $B_1,\ldots,B_k$ is obtained from
a (outer) facial path in $\block_B(G)$ (by \obsref{shit0}).  Again,
\lemref{interleave} implies that the resulting colour sequence is
nonrepetitive.
\end{proof}

\subsection{Colouring Even Cactus Graphs}

We now show how to colour the blocking graph---a cactus graph---of an
outerplane graph.  By Lemma \ref{lem:hitting_plus_four}, if we can find
a facial nonrepetitive $k$-colouring of any cactus graph, we can get a
facial nonrepetitive $k+4$-colouring of any outerplane graph.

Recall that the best known upper bound for the facial Thue chromatic
number of outerplane graphs is 12, which is the bound for the Thue
chromatic number \cite{barat2007square, kundgen2008nonrepetitive}. Thus,
to improve this bound, we need to find a facial nonrepetitive 7-colouring
of the blocking graph. We have been unable to do this unless all cycles of
the blocking graph are even.  We will eventually address this limitation
in \secref{even-blocking-graph} by proving the existence of a blocking
set $B$ such that $\block_B(G)$ has no odd cycles.

For any graph, $G$, a \emph{levelling} of $G$ is a function
$\lambda\colon V(G)\to \{0, 1, 2,\dots\}$ such that for each
$uv\in E(G)$, $|\lambda(u)-\lambda(v)|\leq 1$. The
level pattern of a path $v_1,\ldots,v_k$ is the sequence
$\lambda(v_1),\lambda(v_2),\ldots,\lambda(v_k)$.

\begin{lem}[K{\"u}ndgen and Pelsmajer \cite{kundgen2008nonrepetitive}] \lemlabel{level_pattern_palindrome_free}
 Let $G$ be a graph and $\lambda\colon V(G)\to \{0, 1,
 2,\dots\}$ be a levelling of $G$. Let $S=s_0,s_1,\ldots,s_m$
 be a nonrepetitive palindrome-free sequence on an alphabet
 $\mathcal{A}$ with $m=\max\{\lambda(v) \;|\; v \in V(G)\}$ and
 $c : V(G) \rightarrow \mathcal{A}$ be a colouring of $G$ defined
 as $c(v)=s_{\lambda(v)}$. If a path $P=P_1, P_2$ with
 $|P_1|=|P_2|$ in $G$ is repetitively coloured under $c$, then $P_1$
 and $P_2$ have the same level pattern.
\end{lem}

\begin{lem}\lemlabel{cactus}
  For every cactus graph, $G$, with no odd cycles (and therefore, for
  every blocking graph with no odd cycles), $\pi_f(G)\le 7$.
\end{lem}

\begin{proof}
 By \lemref{non-simple}, we may assume that $G$ is simple.  We may also
 assume that $G$ is connected as this does not affect its nonrepetitive
 chromatic number. Also, assume that $G$ is neither a cycle nor a
 tree since $\pi_f(G) \leq 4 < 7$ for both these classes of graphs.
 If there exists a vertex $v$ of $G$ such that $\deg_G(v)=1$, then let
 the \emph{root} $r$ of $G$ be $v$. Otherwise, let $r$ be any vertex
 of $G$ of degree at least $3$. Let $\lambda$ be a levelling of $G$
 where $\lambda(v)$ is the distance in $G$ from $r$ to $v$. Let $H$
 be a graph that contains all vertices $v\in V(G)$ such that
 \begin{enumerate}
  \item $v$ is on a cycle $C$ of $G$,
  \item $\lambda(v)=\max_{u \in C} \lambda(u)$ and
  \item $\deg_G(v)=2$.
 \end{enumerate} 
 In other words, $H$ contains the vertices of degree 2 that are on the
 deepest level of a cycle (see \figref{cactus-example}). Notice that
 since every cycle of $G$ is even, there is at most one vertex of $H$
 in each cycle of $G$.

 \begin{figure}
    \begin{center}
        \includegraphics{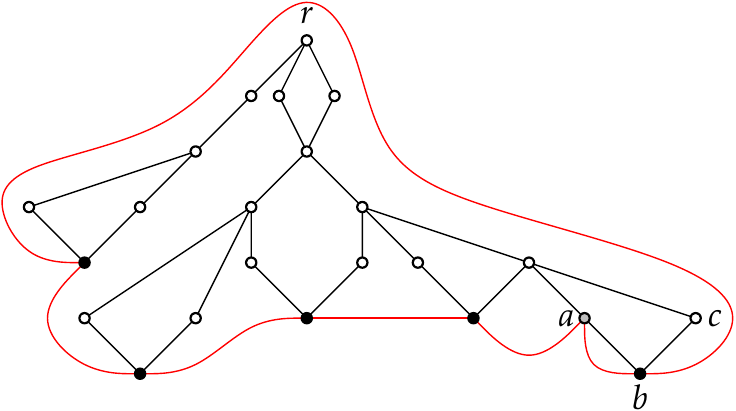}
    \end{center}
    \caption{The graph $H$ (black vertices and grey vertex) obtained
        from a cactus graph with no odd cycles.  The vertex $a$ is added
        in the final step so that $H$ is not a cycle of length 5.}
    \figlabel{cactus-example}
 \end{figure}

If $\deg_G(r)\ne 1$, there must exist at least one face $F^*$ of $G$ such
that exactly one vertex $v$ of $F^*$ has degree greater than two. From
our choice of $r$, it follows that $\lambda(v)$ is the minimum over all
vertices in $V(F^*)$.  Since $|V(F^*)|\ge 4$, $F^*$ has three consecutive
degree-2 vertices $a$, $b$, and $c$, such that $b\in V(H)$. If $|V(H)|\in
\{5,7,10,14,17\}$, we add $a$ to $V(H)$. If $|V(H)|=9$, we add both $a$
and $c$ to $V(H)$. Notice that now, either $\deg_G(r)=1$ or $|V(H)|\notin
\{5,7,9,10,14,17\}$.

We now define the edge set $E(H)$ of $H$.  For each $u,v\in V(H)$,
we add the edge $uv$ to $E(H)$ if there is a facial path on the outer
face of $G$ with endpoints $u$ and $v$ that does not contain any other
vertices in $V(H)$.  Note that $H$ is either a cycle or a forest of
paths. It can only be a cycle if $G$ has no vertices of degree 1,
in which case, $\deg_G(r)\neq 1$.  In this case, our choice of $V(H)$
ensures that the length of this cycle is not in  $\{5,7,9,10,14,17\}$.
This implies that $H$ can be nonrepetitively coloured using the colour set
$\mathcal{B}=\{1,2,3\}$, either by using the result of Thue \cite{thue1906uber}
or Currie (\thmref{cycle}).

To colour the remaining vertices of $G$, let $h=\max_{v \in
V(G)} \lambda(v)$ and $S=s_0,s_1,\ldots,s_h$ be a palindrome-free
nonrepetitive sequence on $\mathcal{A}=\{4,5,6,7\}$.  (A nonrepetitive
palindrome-free sequence can be constructed from any ternary
nonrepetitive sequence by adding a fourth symbol between blocks of size
2 \cite{brevsar2007nonrepetitive}.) Then, each vertex $v\in V(G)\setminus
V(H)$ is assigned the colour $s_{\lambda(v)}$.

We will now show that the resulting 7-colouring of $G$ is a facial
nonrepetitive colouring. Suppose that this is not the case. Thus,
there exists a path $P=P_1,P_2$ such that the colour sequence $S$
corresponding to vertices of $P$ is a repetition. Let us first suppose
that $P$ is on the outer face of $G$. We will need the following claim:

 \begin{clm}\clmlabel{strictly_inc_dec}
    Let $P$ be a path on the outer face of $G$ such that $V(P) \cap V(H)
    = \emptyset$. The level sequence $L$ corresponding to vertices of
    $P$ must be strictly decreasing, strictly increasing, or strictly
    decreasing then strictly increasing.
 \end{clm}
 \begin{proof}[Proof of \clmref{strictly_inc_dec}]
   Suppose that this is not the case. Then $L$ cannot contain two
   consecutive elements of the form $i,i$ as this can only correspond
   to an odd cycle of $G$, but all cycles of $G$ are even. Thus, $L$
   must contain a block of the form $i,i+1,i$. Since $P$ is on the outer
   face, we must have that the vertex $v$ corresponding to $i+1$ is the
   highest numbered vertex on some cycle $C$ and that $\deg_G(v)=2$. But
   in this case, $v$ must be in $H$, which is a contradiction.
 \end{proof}
 By Lemma \ref{lem:level_pattern_palindrome_free}, $P_1$ and $P_2$ have
 the same level pattern. However, if $V(P) \cap V(H) = \emptyset$ this
 is incompatible with \clmref{strictly_inc_dec}.  Thus, $P$ must contain
 vertices of $H$. Let $P_H=p_1,p_2,\ldots,p_k$ be the sequence of vertices
 of $V(P) \cap V(H)$ in the same order as in $P$. Notice that $P_H$ is
 a path in $H$.  Therefore, the colour sequence corresponding to $P_H$
 is nonrepetitive.  Now, observe that the colour sequence formed by $P$
 is of the form $A_0,B_1,A_1,B_\ell,A_\ell$ where $B_1,\ldots,B_\ell$
 is a non-repetitive sequence of colours from $\mathcal{B}=\{1,2,3\}$ and
 each $A_i$ is a non-repetitive sequence of colours from $\mathcal{A}=\{4,5,6,7\}$.
 Therefore, by \lemref{interleave}, $P$ is coloured nonrepetitively.

Thus, $P$ must be a facial path on some inner face $F$ of $G$.  If $V(P) \cap V(H) =
\emptyset$ then, by Lemma \ref{lem:level_pattern_palindrome_free},
$P_1$ and $P_2$ have the same level pattern.  No path on an even cycle
has such a pattern using the levelling $\lambda$.
Therefore, $V(P)\cap V(H)\neq\emptyset$, so $P$ contains 1, 2, or
3 vertices of $H$.  If $P$ is a repetition it must contain exactly
2 vertices of $H$, thus $P$ is a facial path on $F^*$ (since every
other inner face contributes at most one vertex to $V(H)$).  Reusing the
notation above, $P$ cannot contain $b$ since $b$ has a unique colour in
$V(F^*)$. Therefore, $P$ must contain $a$ and $c$ and, in fact, these are
the endpoints of $P$.  The colour sequence of $P$ must therefore be of
the form $xAx$ where $x\in\{1,2,3\}$ and $A$ is a non-empty sequence over
the alphabet $\{5,6,7,8\}$.  Such a sequence is not a repetition.
\end{proof}

\subsection{Making an Even Blocking Graph}
\seclabel{even-blocking-graph}

\begin{lem}\lemlabel{even-biconnected}
  For every biconnected outerplane graph, $G$, and any vertex $v\in V(G)$:
  \begin{itemize}
    \item $G$ has a blocking set $B$ such that $\block_B(G)$ is an even cycle 
       and $v\in B$; and
    \item $G$ has a blocking set $B$ such that $\block_B(G)$ is an even cycle 
       and $v\not\in B$.
  \end{itemize}
\end{lem}

\begin{proof}
  We first obtain a blocking set $B'$ that contains or does not
  contain $v$, as appropriate, by applying \lemref{biconnected} or
  \corref{biconnected-out}. Recall that $B'$ contains exactly one vertex
  on each inner face of $G$.  It is simple to verify that $\block_{B'}(G)$
  is a cycle;  if it is an even cycle, then we are done, so we may assume
  that $\block_{B'}(G)$ is an odd cycle.

  If $G$ has only one inner face, then $B'$ contains one vertex,
  $u$, on this face and $\block_{B'}(G)$ is an odd cycle (of length 1).
  In this case, we can select the neighbour, $w$, of $u$ such that
  $w\neq v$ and let $B=B'\cup\{w\}$.  It is easy to verify that $B$ is
  a blocking set, and either includes or excludes $v$, as appropriate,
  and that $\block_B(G)$ is a cycle of length 2.

  Thus, we may assume that $G$ has at least two inner faces and we
  consider several cases:
 
  \begin{enumerate}
  \item If $G$ contains an ear, $F$, with four or more vertices such that
  either $v\not\in V(F)$ or $v$ is one of the endpoints of the chord of
  $F$. There is exactly one vertex $x\in B'$ on the face $F$.  Let $y$
  be a neighbour of $x$ on $F$ such that $y$ is not on the chord of $F$
  (so $y$ has degree 2). Such a $y$ exists since $F$ has at least four
  vertices.  Set $B=B'\cup \{y\}$.  Now $|B|$ is even, so $\block_B(G)$
  is an even cycle.  Furthermore, $G-B'$ is a tree and $y$ is a leaf in
  this tree, so $G-B$ is also a tree.  Finally, by choice, $B$ contains
  $v$ if and only if $B'$ contains $v$, so $B$ satisifies the conditions
  of the lemma.

  \item Next, consider the case where $G$ contains a triangular ear, $F$,
  such that one of the endpoints of the chord of $F$ is in $B'$ and $v$
  is not the degree 2 vertex, $y$, of $F$.  By the same argument as above,
  $B=B'\cup\{y\}$ satisfies the conditions of the lemma.

  \item Refer to \figref{toughie}.  For an edge $uw\in E(G)$, let
  $V_{uw}'$ and $V_{uw}''$ be the two (possibly empty) sets of vertices
  in the (at most) two connected components of $G-\{u,w\}$.  Let
  $G_{uw}'=G[V_{uw}'\cup\{u,w\}]$ and $G_{uw}''=G[V_{uw}''\cup\{u,w\}]$.
  If neither of the two previous  cases applies, then there exists an
  edge $uw$ of $G$ such that $v\not\in V_{uw}'$ and the weak dual of
  $G_{uw}'$ is a star whose central vertex is the face, $F_{uw}$ incident
  on $uw$.  

  \begin{figure}
    \begin{center}
       \includegraphics{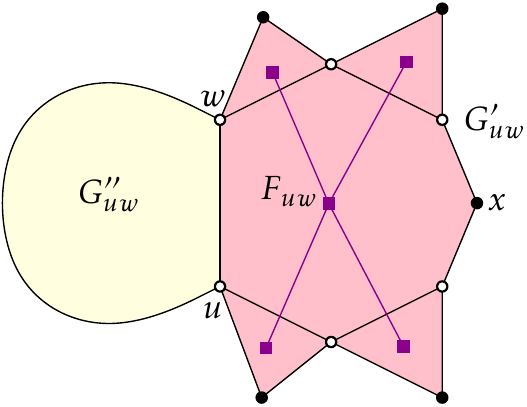}
    \end{center}
    \caption{The face $F_{uw}''$ in the proof of \lemref{even-biconnected}.}
    \figlabel{toughie}
  \end{figure}

  We now argue why such an edge $uw$ exists.  Recall that the weak dual,
  $G^{\circ}$, of $G$ is a tree whose vertices are the faces of $G$.
  Select some face, $R$, of $G$ that has $v$ on its boundary and root
  $G^\circ$ at $R$. This tree has a height, $h$, and some vertex $F$ of
  depth $h-1$ (recall that $G$ has at least two inner faces).  The face
  $F$ will be the face $F_{uw}$ described above.  We now show how to
  choose the edge $uw$.

  If $F=R$ (because $h=1$), then we take $uw$ to be an edge of $F$,
  one of whose endpoints is $v$.  Such an edge $uw$ exists since $v$
  is on $F$.  (Note that, in this case, $uw$ may be a chord or may be
  on the outer face.)

  If $F\neq R$, then we take $uw$ to be the chord of $F$ that separates
  it from $R$.  (In this case, $v$ may still be one of $u$ or $w$.)
  In either case, the edge $uw$ and the face $F=F_{uw}$ satisfy the
  condition described above.  In particular, the dual of $G'_{uw}$ is a
  star because $F$ had height $h-1$ and $v\not\in V'_{uw}$ by our choice
  of $uw$.

  Now that we have established the existence of $uw$ and $F_{uw}$,
  we will now show that we can select another vertex, $y$, from
  $V_{uw}'\setminus B'$ so that $B=B'\cup\{y\}$ is a blocking set.
  This is sufficient since $|B|$ is even so $\block_B(G)$ is an even
  cycle.  

  \begin{itemize}
  \item[$(\star)$]
  By choice, $G_{uw}'$ has at least 2 faces and each of them, other
  than $F_{uw}$, is a triangular ear incident to $F_{uw}$ and whose
  degree-2 vertex is in $B'$ (otherwise, one of those faces would have
  been handled by Case~1 or 2).
  \end{itemize}

  Let $x$ be the unique vertex of $B'$ on $F_{uw}$. The vertex $x$ has two
  neighbours on $F_{uw}$.  We claim that one of these is not in $\{u,w\}$
  and we take this vertex to be $y$.  This claim is valid because
  otherwise, $F_{uw}$ is a triangle, $xuw$, with $x\not\in\{u,w\}$.
  This case is not possible because by $(\star)$ at least one of $xu$ or
  $xw$ is incident on both $F_{uw}$ and a triangular ear $E$ of $G_{uw}'$.
  Both $x$ and the degree 2 vertex of $E$ are in $B'$.  This contradicts
  the fact that $B'$ includes at most one vertex from each face of $G$,
  including $E$.

  Let $B=B'\cup\{y\}$.  We claim that $B$ is a blocking set of $G$.
  First, note that $B$ contains at most two vertices from each face, $F$,
  of $G$, so $V(F)\setminus B'\neq \emptyset$.  We now show that $y$ is a
  leaf in the tree $G-B'$, so that $G-B$ is also a tree.

  First, observe that $xy$ is not a chord of $G$ since, by $(\star)$,
  the face incident to $xy$ other than $F_{uw}$ would have two of its
  vertices in $B'$.  Thus, in addition to $x$, $y$ has at most two
  neighbours in $G$.  One of these, $z$, is on $F_{uw}$ and $z\neq x$
  so $z\notin B'$.  Finally, $y$ may have one additional neighbour,
  which is a degree-2 vertex of a triangular ear incident on $yz$.
  In this case, by $(\star)$, this degree-2 vertex is in $B'$.  Thus,
  $yz$ is the only edge incident to $y$ in the tree $G-B'$
  so $y$ is a leaf in this tree. \qedhere
\end{enumerate}
\end{proof}

\begin{rem}\remlabel{bw}
   The proof of \lemref{even-biconnected} can be modified to prove
   something stronger than just requiring the inclusion or exclusion
   of $v\in B$.  We can specify an edge $ab$ on the outer face of $G$
   and obtain a blocking set $B$ such that $\block_B(G)$ is an even
   cycle, $a\not\in B$ and $b\in B$.  The only difference in the proof
   is ensuring that $a$ is not included in $B$. The resulting proof
   has the same three cases. Case~1 applies as long as $ab$ is not on
   the boundary of the ear $F$. Case~2 applies as long as $ab$ is not
   on the boundary of the ear $E$.  Otherwise, in Case~3, the edge $ab$
   is in the subgraph $G_{uw}''$, so there is no chance of including $a$
   in $B$.  This stronger version of \lemref{even-biconnected} is used
   in \appref{biconnected}.
\end{rem}

\begin{lem}\lemlabel{even-bridgeless}
  Every simple bridgeless outerplane graph $G$ has a blocking set $B$
  such that all cycles in $\block_B(G)$ are even.
\end{lem}

\begin{proof}
  The proof is by induction on the number of 2-connected components
  of $G$.  If $G$ has no 2-connected components, then we take $B$ to be
  the empty blocking set.  If $G$ has only one 2-connected component,
  then we apply \lemref{even-biconnected}.
  Otherwise, select a 2-connected component, $C$, that
  shares exactly one vertex, $v$, with the rest of $G$.  Let
  $G'=G-(V(C)\setminus\{v\})$ and apply induction on $G'$
  to obtain a blocking set, $B'$, of $G'$ such that $\block_{B'}(G')$
  has only even cycles.  There are two cases to consider:
  \begin{enumerate}
    \item If $B'$ contains $v$, then we apply the first part of
    \lemref{even-biconnected} to obtain a blocking set $B''$ of $C$
    such that $\block_{B''}(C)$ is an even cycle and $v\in B''$.  We take
    $B=B'\cup B''$, which clearly forms a blocking set of $G$.  
    The blocking graph $\block_B(G)$ is simply the union of the
    two blocking graphs $\block_{B'}(G)$ and $\block_{B''}(C)$, which have
    only the vertex $v$ in common.  Thus, every cycle in $\block_B(G)$
    is also a cycle in one of these two graphs, so it has even length.

    \item If $B'$ does not contain $v$, then we apply the second part
    of \lemref{even-biconnected} to obtain a blocking set $B''$ of $C$
    such that $\block_{B''}(C)$ is an even cycle and $v\not\in B''$.
    We take $B=B'\cup B''$.

    Refer to \figref{block-merge}.  
    Starting at some appropriate vertex in $V(G)\setminus V(C)$
    in the facial walk on the outer face of $G$, there is a last vertex,
    $x\in V(G')\cap B'$, encountered before the walk encounters the first
    vertex $u\in V(C)\cap B''$ and there is a last vertex $w\in V(C)\cap
    B''$ encountered before the walk returns to the next vertex $y\in
    V(G')\cap B'$.  The edge $xy$ is in $\block_{B'}(G')$ and the edge
    $vw$ is in $\block_{B''}(C)$.

    \begin{figure}
       \begin{center}
         \begin{tabular}{m{.45\textwidth}m{.05\textwidth}m{.45\textwidth}}
          \includegraphics[width=.45\textwidth]{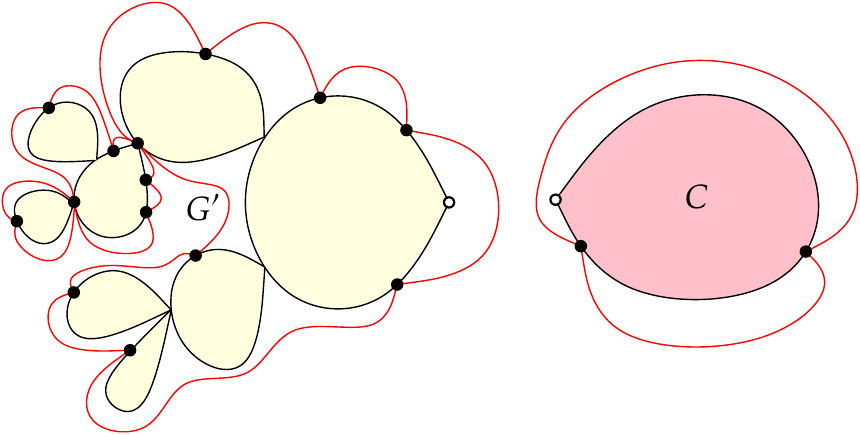} &
          $\Rightarrow$ &
          \includegraphics[width=.45\textwidth]{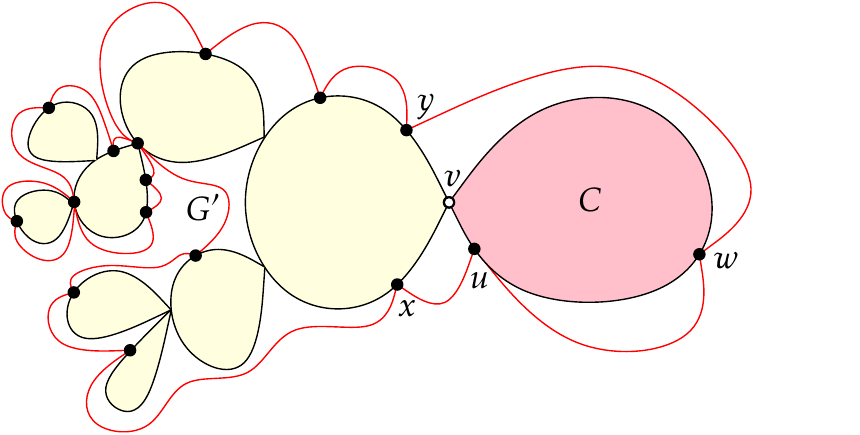}
         \end{tabular}
       \end{center}
       \caption{Case~2 in the proof of \lemref{even-bridgeless}.}
       \figlabel{block-merge}
    \end{figure}

    Since every blocking graph is a bridgeless cactus graph
    (\obsref{cactus}), each of these edges is part of one even cycle
    in its respective graph. In $\block_{B}(G)$ these two cycles are
    merged by removing the edges $xy$ and $vw$ and adding the edges
    $xv$ and $yw$.  The resulting cycle is even.  Every other cycle
    in $\block_B(G)$ is also a cycle in one of $\block_{B'}(G')$ or
    $\block_{B''}(C)$ so it has even length. \qedhere
  \end{enumerate}
\end{proof}

\begin{lem}\lemlabel{even}
  Every simple outerplane graph $G$ has a blocking set $B$ such that
  all cycles in $\block_B(G)$ are even.
\end{lem}

\begin{proof}
  The proof is by induction on the number of bridges of $G$.  If $G$
  has no bridges, then we apply \lemref{even-bridgeless}.  Otherwise,
  select some bridge $uw$ of $G$ and contract it to obtain a graph $G'$
  in which $uw$ corresponds to a single vertex $v$.  By induction,
  we obtain a blocking set $B'$ of $G'$ such that $\block_{B'}(G')$
  has only even cycles (or is empty).  There are two cases to consider:
  \begin{enumerate}
    \item If $v\in B'$, then we take $B=B'\cup\{u,w\}\setminus\{v\}$. This introduces exactly one new cycle in $\block_B(G)$ that is not present in $\block_{B'}(G')$ and this cycle has length 2.

    \item If $v\not\in B'$, then we take $B=B'$, so
    $\block_B(G)=\block_{B'}(G')$. \qedhere
  \end{enumerate}
\end{proof}

Finally, have all the tools to prove our main result on outerplane graphs:

\begin{thm}\thmlabel{outerplane}
  For every outerplane graph, $G$, $\pi_f(G)\le 11$.
\end{thm}

\begin{proof}
By \lemref{non-simple}, we may assume that $G$ is simple.  From
\lemref{even}, $G$ has a blocking set $B$ such that $\block_B(G)$
has no odd cycles.  Therefore, by \lemref{cactus}, $\pi_f(\block_B(G))\le
7$.  Using this with \lemref{k-plus-four} implies that $\pi_f(G)\le 11$.
\end{proof}

\section{Plane Graphs}

In this section, we show how to reuse the ideas from Bar\'at and Czap
\cite{barat2013facial} to facially nonrepetitively 22-colour every plane graph.
Some modifications are needed because Bar\'at and Czap use a nonrepetitive
12-colouring of outerplanar graphs whereas our \thmref{outerplane} provides an
\emph{facial} nonrepetitive 11-colouring of \emph{outerplane} graphs.

\begin{thm}
   Let $r=\max\{\pi_f(H):\text{$H$ is outerplane}\}$.  Then, for every
   plane graph $G$, $\pi_f(G)\le 2r$.
\end{thm}

\begin{proof}
   For any plane graph, $H$, the \emph{peeling layering} of $H$ is
   a partition of $V(H)$ into sets as follows.  Let $V_0(H)$ be the
   vertices on the outer face of $H$ and let $V_i(H)$, $i\ge 1$, be the
   vertices on the outer face of $H-(V_0(H)\cup\cdots\cup V_{i-1}(H))$.

   We augment $G$ to obtain a plane graph $G^+$ in the following way.
   For each inner face $F$ of $G$, let $W$ be the facial walk of $F$.
   The walk $W$ contains only vertices from $V_i(G)$ and $V_{i+1}(G)$
   for some $i$.  Remove from $W$ all vertices in $V_{i+1}(G)$ to
   obtain a cyclic sequence $W'$ of vertices from $V_i(G)$.  For any
   two consecutive vertices $u,w$ in $W'$, we add the edge $uw$
   to $G^+$ and embed it inside the face $F$. This construction has
   the following implications: (a)~The resulting graph $G^+$
   is still plane (though not necessarily simple) and, for all
   $i$, $V_i(G)=V_i(G^+)$. From this point onward, we use the notation
   $V_i=V_i(G)=V_i(G^+)$.
   (b)~The cyclic sequence $W'$ defined above is a facial walk
   in $G^+$ and, since it contains only vertices in $V_i$, it is also
   a facial walk in $G^+[V_i]$.

   For each $i$, $G^+[V_i]$ is outerplane.  To colour $G$ we use
   \thmref{outerplane} to facially nonreptitively colour $G^+[V_i]$ with
   $\{1,\ldots,11\}$ if $i$ is even or $\{12,\ldots,22\}$ if $i$ is odd.
   This defines a colouring of $G$ that we now prove is facially
   nonrepetitive.

   Let $P$ be a facial path, $F$, in $G$. The graphs $G^+[V_0]$ and $G$
   have the same outer face so, if $P$ is on the outer face,
   then the colour sequence of $P$ is not a repetition
   since our colouring is facially nonrepetitive for $G^+[V_0]$.

   Therefore, $F$ is an inner face and all vertices on $F$ are in $V_i\cup
   V_{i+1}$ for some $i$.  Write $P$ as $P_0,Q_1,P_1,\ldots,Q_k,P_k$,
   where each $Q_j$ consists of vertices from $V_i$ and each $P_j$
   consists of vertices from $V_{i+1}$.  Notice that, for each
   $j\in\{1,\ldots,k-1\}$, $G^+[V_i]$ contains an edge joining the
   last vertex in $Q_j$ to the first vertex in $Q_{j+1}$.  Indeed,
   $Q_{1},\ldots,Q_k$ is a facial path in $G^+[V_i]$ (It is a contiguous
   subsequence of the sequence $W'$ defined above.)  Next, observe that
   each $P_j$ is a facial path on the outer face of $G^+[V_{i+1}]$.
   Therefore, the colour sequence determined by $P$ is of the form
   $A_0,B_1,A_1,\ldots,B_k,A_k$ (with $A_j$ corresponding to $P_j$
   and $B_j$ corresponding to $Q_j$).  The sequence $B_1,\ldots,B_k$ is
   nonrepetitive and each sequence $A_j$ is non-repetitive.  Therefore,
   by \lemref{interleave}, the colour sequence determined by $P$
   is nonrepetitive.
\end{proof}

\section{Concluding Remarks}

We note that the proofs in this paper lead to straightforward linear-time
algorithms.  After the appropriate decomposition steps, there are
essentially two subproblems: (1)~finding an appropriate blocking
set in a biconnected outerplane graph (\lemref{even-biconnected}) and
(2)~colouring cactus graphs with no odd cycles (\lemref{cactus}). The
proof of \lemref{cactus} is easily made into a linear-time algorithm. The
proof of \lemref{even-biconnected} can be implemented by a recursive
ear-cutting algorithm that implements \lemref{biconnected} followed by
a traversal of the dual tree in order to find the face $F_{uw}$ used in
the proof of \lemref{even-biconnected}.

It seems unlikely that our upper bound of 11 for the facial nonrepetitive
chromatic number of outerplane graphs (and hence the bound of 22 for
plane graphs) is tight.  (Recall that the best known lower bounds
are 4 and 5, respectively \cite{barat2013facial}.)  Thus, an obvious
direction for future work is to improve these bounds.  Our proof of
\lemref{hitting_plus_four} uses a nonrepetitive 4-colouring of trees, but a
facial nonrepetitive colouring would also be sufficient.  This leads
naturally to the following problem:

\begin{op}
   Is $\pi_f(T) \le 3$ for every tree, $T$? 
\end{op}

\bibliographystyle{plain}
\bibliography{facial}

\appendix
\section{Biconnected Outerplane Graphs}
\applabel{biconnected}

\begin{lem}\lemlabel{good-one}
  If $G$ is a biconnected outerplane graph, then $G$ has a blocking set $B$
  such that $|B|\not\in\{5,7,9,10,14,17\}$.
\end{lem}

\begin{proof}
  If $G$ is a cycle, take $B$ to be two consecutive vertices on $G$
  and we are done.
  Otherwise, select an ear, $E$ of $G$ and let $uv$ be $E$'s chord.
  Then apply the stronger version of \lemref{even} discussed in
  \remref{bw} to the graph $G'=G-(V(E)\setminus\{u,v\})$ to obtain a
  blocking set $B'$ of even size that contains $v$ and not $u$.

  Note that, since $v\in B'$ and $v$ is on $E$,  $G-B'$ has no cycles.
  Furthermore, since $u\not\in B'$, $G-B'$ is connected, so $B'$
  is a blocking set of $G$.  Since $|B'|$ is even, 
  $|B'|\not\in\{5,7,9,17\}$, so if $|B'|\not\in\{10, 14\}$, then we are
  done with $B=B'$. Otherwise, let $y$ be $v$'s degree-2 neighbour on $E$
  and take $B=B'\cup\{y\}$.
\end{proof}

\begin{cor}
  If $G$ is an outerplane graph with at most one 2-connected component, 
  then $\pi_f(G)\le 7$.
\end{cor}

\begin{proof}
  If $G$ is a tree, then $\pi_f(G)\le 4$, so we may assume $G$ contains
  exactly one 2-connected component, $G'$.  Apply \lemref{good-one} to
  $G'$ to obtain a blocking set $B$ with $|B|\not\in\{5,7,9,10,14,17\}$
  and observe that $B$ is also a blocking set of $G$.  The blocking graph
  $\block_B(G')$ is a cycle, $C$, that has a nonrepetitive 3-colouring.
  The blocking graph $\block_B(G)$ consists of $C$ and possibly some
  self loops so, by \lemref{non-simple}, $\block_B(G)$ has a facial
  nonrepetitive 3-colouring.  Therefore, by \lemref{hitting_plus_four},
  $\pi_f(G)\le 7$.
\end{proof}

\end{document}